\documentclass[reqno]{amsart}
\usepackage{epsfig,graphicx,amssymb,amscd}
\DeclareMathOperator{\dm}{dom}
\DeclareMathOperator{\im}{img}

\newcommand{\rl}{{\mathbb{R}}}
\newcommand{\cx}{{\mathbb{C}}}
\newcommand{\dbar}{\overline{\partial}}

\newcommand{\csor}{{\widehat{\otimes}}}
\newcommand{\tensor}{\otimes}

\newcommand{\pss}{\widetilde{W}}
\newcommand{\smooth}{\mathcal{C}^\infty}
\newcommand{\der}{\mathsf{D}}

\newcommand{\abs}[1]{\left|{#1}\right|}
\newcommand{\norm}[1]{\left\|{#1}\right\|}
\newcommand{\length}[1]{\left[#1\right]}

\newtheorem{thm}{Theorem}[section]
\newtheorem{lem}[thm]{Lemma}

\newtheorem{cor}[thm]{Corollary}

\begin{document}
\title{ The Cauchy-Riemann equations on   product domains}
\author{Debraj Chakrabarti}
\address{Department of Mathematics,  University of Notre Dame,
Notre Dame, IN 46556, USA} \email{dchakrab@nd.edu}
\author{Mei-Chi Shaw}
\address{Department of Mathematics,  University of Notre Dame,
Notre Dame, IN 46556, USA} \email{mei-chi.shaw.1@nd.edu}
\begin{abstract} We  establish the $L^2$ theory for   the Cauchy-Riemann equations 
on product domains provided that the Cauchy-Riemann operator has closed range on each factor.
We deduce regularity of the canonical solution on $(p,1)$-forms in  special Sobolev spaces represented
as tensor products of Sobolev spaces on the factors of the product. This leads to regularity 
results for smooth data.

\end{abstract}
\maketitle
\maketitle
\section{Introduction}

In this paper we study the existence and regularity for the solution of the inhomogeneous Cauchy-Riemann equations,
or the $\dbar$-equation  on product domains.  
When the product domain  is a polydisc in $\cx^n$, 
the solution to the  $\dbar$-equation can be obtained
by an inductive process from the solution in one variable 
given by the  Cauchy integral formula for the disc.
This  is known as the Dolbeault-Grothendieck Lemma 
(see \cite[Theorem~2.1.6]{cs}. For other approaches, see \cite{nick,nw}.)
which is the analog for the $\dbar$-operator
of the Poincar\'{e} lemma for the exterior derivative $d$.

We are interested here in the $\dbar$-problem in the $L^2$ setting.
For a bounded pseudoconcovex domain in $\cx^n$, or more generally in a Stein manifold,
$L^2$ existence theorems have been established in H\"ormander \cite{Hor1}. We prove
$L^2$ existence on a product, i.e., we show that $\dbar$ has closed range on a product
provided that $\dbar$ has closed range on  each factor domain . 

\begin{thm}\label{thm-main} For $j=1,...,N$, let $\Omega_j$ be a  
relatively compact domain with Lipschitz boundary in a complex hermitian manifold  $M_j$ .   
Let $\Omega\subset M_1\times \dots\times M_N$  be the product domain
$\Omega=\Omega_1\times\dots\times\Omega_N$. Suppose the  $\dbar$ 
operator has closed range in $L^2(\Omega_j)$ for all degrees  for each $j$,
then the $\dbar$ operator has closed range for all degrees  in $L^2(\Omega)$. 
Furthermore,  the  K\"{u}nneth formula holds for the $L^2$ cohomology:
\[   H^*_{L^2}(\Omega)= H^*_{L^2}(\Omega_1)\csor \dots\ \csor H^*_{L^2}(\Omega_N),\] 
where   $\csor$ denotes the Hilbert space tensor product. 
\end{thm}
If the $L^2$ space on the domain $\Omega_j$ is defined with respect to a weight function
$\phi_j$, i.e., if on $\Omega_j$ we use the norm $\int_{\Omega_j}\abs{f}^2e^{-\phi_j}dV$,
the same statement holds if $\sum_{j=1}^N\phi_j$ is used as a weight on the product $\Omega$.

A classical approach to the study of partial differential equations on product domains
is by separation of variables and spectral representation. This method
can be applied to the $\Box$-operator (the complex Laplacian $\dbar\dbar^*+\dbar^*\dbar$):
see \cite[pp. 103ff]{gh} for the case 
of compact complex manifolds and \cite{fu} for the case of the polydisc. A proof of
a general version of Theorem~\ref{thm-main} (without the assumption of relative compactness or boundary regularity of the $\Omega_j$) may be given using separation of variables and spectral theory (see \cite{cproc}.) However,
it is difficult to use this method to draw conclusions about the regularity of the 
solution of the $\dbar$-equation.
There is a different approach, using a direct construction of a solution operator on 
a product domain,
used first in \cite{zuck} for the de Rham complex,
and we use this approach to prove Theorem~\ref{thm-main}.
Using this method we not only prove Theorem~\ref{thm-main},
but also obtain regularity results for the canonical solution of $\dbar$.

 The closed-range property given by Theorem~\ref{thm-main}
 has numerous applications.  First it immediately gives  that the Hodge decomposition holds
for the product domain $\Omega$. Notice that it is not assumed that 
the domains $\Omega_j$ are non-compact:  the  theorem can be applied to the case when the product
domain is a product $D\times M$  of a bounded  pseudoconvex domain $D$  in $\cx^n$ and  a compact complex manifold  $M$. 
Though the proof for Theorem~\ref{thm-main} is not difficult, it has not been stated explicitly in the literature.

We obtain boundary regularity results  for the  canonical solution of the $\dbar$-equation on  product domains in
$\cx^n$ or complex hermitian manifolds. The regularity for the canonical solution of the $\dbar$-equation and the $\dbar$-Neumann operator on a polydisc have been studied extensively 
(see \cite{eh1,eh2,eh3,jb} and the references in these works.)
There is also a considerable amount of work  for the $\dbar$-equation 
on domains with Lipschitz boundary or piecewise smooth domains (see \cite{ms4}). 
Notice that a product domain is only
{\em piecewise} smooth even if each factor domain has smooth boundary.  Thus  the boundary is only  Lipschitz. It is known that on a general Lipschitz domain,  the  $\dbar$-Neumann operator or even the Green's operator for the Dirichlet problem (see \cite{BV,mcs2})  is not regular near the singular part of the domain.
Thus one cannot expect  the $\dbar$-Neumann operator to be regular 
near  the product of the boundaries of the factor domains. This is confirmed
by the explicit computations in \cite{eh1,eh2,eh3}. One would expect that the canonical solution
might also not be regular. An interesting feature is that while the $\dbar$-Neumann
operator on product domains might  not be  well-behaved, the canonical solution still exhibits
regularity on certain Sobolev spaces.  Before our results here, only $\mathcal{C}^k$ estimates were
known   for the special case of the   polydisc  using an explicit integral formula \cite{lan}.

In order to state precise regularity results on the canonical solution operator 
we introduce  special Sobolev spaces, called {\em Partial Sobolev spaces},
denoted by $\pss^k(\Omega)$  (for  definition of $\pss^k(\Omega)$, see $\S$\ref{sec-partialsobolev}.)
If $W^k(\Omega)$ denotes the usual Sobolev space of functions having $L^2$-derivatives of order $k$
on $\Omega$, we have $W^{Nk}(\Omega)\subset \pss^k(\Omega)\subset W^k(\Omega)$.
We prove the following regularity result for the canonical
solution operator  in the  partial  Sobolev spaces  on a product  pseudoconvex domain. 
\begin{thm}\label{thm-reg}
Let $\Omega$ be the same as in Theorem~\ref{thm-main}. Then the $\dbar$-Neumann operator  $\mathsf{N}$ exists for  all degrees on  $(p,q)$-forms with $L^2$ coefficients.   Assume further that for each $j$, the domain
$\Omega_j$ is smoothly bounded, and  the $\dbar$-Neumann
operator on  $\Omega_j$ preserves the space of forms with coefficients in $W^k(\Omega_j)$  for every integer $k\geq 0$.
For any $p$ with $0\leq p\leq \dim_\cx\Omega$,  let $f$ be a  $\dbar$-closed $(p,1)$-form on $\Omega$
orthogonal to the $(p,1)$-harmonic forms such that
the coefficients of $f$ are in the partial Sobolev space $\pss^l(\Omega)$, for some integer $l\geq 0$. Then, the canonical
solution $u=\dbar^* \mathsf{N}f$ of the equation $\dbar u=f$ also has coefficients in $\pss^l(\Omega)$.
\end{thm}
As with Theorem~\ref{thm-main}, 
if the $L^2$ space on $\Omega_j$ is defined with respect to a weight function $\phi_j$, 
the same conclusion holds if the $\dbar$-Neumann operator $\mathsf{N}_{\phi_j}$ with weight $\phi_j$ 
preserves $W^k(\Omega_j)$ forms, and the canonical solution on the product is taken with 
respect to  the weight $\sum_{j=1}^N\phi_j$.  Note that it follows  from the 
inclusions  $W^{Nk}(\Omega)\subset \pss^k(\Omega)\subset W^k(\Omega)$ 
 that if the $(p,1)$-form  $f$ has coefficients
in $W^{Nk}(\Omega)$, then the canonical solution $\dbar^* \mathsf{N}f$  has coefficients in 
$W^k(\Omega)$. Of course this loss of smoothness  disappears on using the correct
space $\pss^k(\Omega)$. Also note that  unless $D$ is a domain in $\cx^n$, the $\dbar$-equation
for $(p,q)$ forms on a domain $D$ in a complex manifold cannot be reduced to the $\dbar$-equation for $(0,q)$ forms.

To use this result, we need to understand the regularity of the $\dbar$-Neumann operator
on the factor domains. There is a vast  literature on   the  regularity of the
$\dbar$-Neumann operator on smooth and pseudoconvex domains.  In particular, regularity
is known when the boundary is strongly pseudoconvex (see \cite{Ko1}) or finite type (see \cite{Ca}),
or if the boundary has a plurisubharmonic defining function (see \cite{BS}) or if the boundary
has transverse symmetry ( see \cite{ba1}.)  However, for each $s>0$, there exists a 
pseudoconvex domain with smooth boundary such that   the $\dbar$-Neumann operator or the
canonical solution is not regular in the  Sobolev space $W^s$ (see \cite{ba2}). 
Even in this case, we can obtain regularity in a weighted Sobolev space (see \cite{Ko2}.)  Using
these results, one can draw many corollaries from Theorem~\ref{thm-reg} regarding the regularity of the 
solution of the $\dbar$-problem in Sobolev spaces or spaces of smooth forms
(see Corollary~\ref{cor-main} below.) One example is the following:
\begin{cor}\label{cor-strong}
Suppose that the smoothly bounded pseudoconvex domains  $\Omega_1,\dots,\Omega_N$ in hermitian 
manifolds of dimension $n_1,\dots,n_j$ respectively are such that for each $j$, and every $0\leq p\leq n_j$,
the canonical solution operator  on $\Omega_j$ maps the space $\mathcal{C}^\infty_{p,1}(\overline{\Omega_j})$
of $(p,1)$ forms  smooth up to the boundary to $\mathcal{C}^\infty_{p,0}(\overline{\Omega_j})$.
Let $\Omega=\Omega_1\times\dots\times\Omega_N$.  For $0\leq \boldsymbol{p}\leq\sum_{j=1}^N n_j$, 
let $f$ be a $\dbar$-closed $(\boldsymbol{p},1)$ form with $\smooth(\overline{\Omega})$ coefficients.
Then $\dbar^*\mathsf{N}f$ also has coefficients in $\smooth(\overline{\Omega})$, where $\mathsf{N}$
is the $\dbar$-Neumann operator on $\Omega$.
Further, the Bergman projection on the space of functions $L^2(\Omega)$ preserves the space $\smooth(\overline{\Omega})$.
\end{cor}
Note that if $f$ is orthogonal to the harmonic forms, $\dbar^*\mathsf{N}f$ is the canonical solution
to $\dbar u=f$. Also, Corollary~\ref{cor-strong} applies to the product of domains which are { strongly} pseudoconvex or
more generally of finite type. In contrast, we note that  when the domain is the intersection of two balls, 
the Bergman projection is not regular near the  nongeneric points of the boundary 
(see \cite{BV}).

Notice that on a smoothly bounded pseudoconvex domain in $\cx^n$,
if the canonical solution    $\dbar^*\mathsf{N}$ is regular,
it follows that   the  $\dbar$-Neumann operator  $\mathsf{N}$, 
and the  adjoint of the canonical solution operator 
$\dbar\mathsf{N}$ are all exact  regular on Sobolev spaces (see \cite{cs}.)
However, the same method cannot be applied to the 
adjoint of the canonical solution or to the $\dbar$-Neumann operator on a  product domain. 
On a product domain,    the canonical solution is regular,   but neither  the $\dbar$-Neumann operator
$\mathsf{N}$ 
nor the operator $\dbar\mathsf{N}$  is  regular    near the boundary.

The plan of this paper is as follows: in $\S$\ref{sec-l2setting} and in $\S$\ref{sec-diffforms} we establish
terminology and notation regarding the $L^2$ $\dbar$-problem and tensor products of forms,
and discuss some basic properties of the objects involved.
We note here that although our results have been stated for general manifolds, in these sections,
for simplicity of exposition and notation,
we give the definitions  for domains in Euclidean space $\cx^n$. The generalization to manifolds is 
easy and left to the reader. Also, due to the nature of the proof, we need to consider spaces 
of forms of arbitrary degrees. For example, we denote by $L^2_*(D)$ the space of forms with square 
integrable coefficients on a domain $D$. The key observation in $\S$\ref{sec-l2setting} is that
closure of the range of the $\dbar$-operator is a necessary as well as sufficient condition for
representation of cohomology classes by harmonic forms (see  Lemma~\ref{lem-j}.) The next $\S$\ref{sec-dbarl2}
represents the central argument of the paper.  Starting from the canonical solution operator and
the harmonic projection on the factor domains, we write down a formula \eqref{eq-tnj} 
defining a solution operator  $S$ on the product, which coincides with the canonical solution operator
$\dbar^*\mathsf{N}$ on $(0,1)$-forms. 
Using $S$ we give a simple proof of Theorem~\ref{thm-main}. In $\S$\ref{sec-partialsobolev} we consider
the tensor products of Sobolev spaces, which gives rise to the partial Sobolev spaces
referred to above. This is used in the last $\S$\ref{sec-regularity} to prove regularity results.

{\em Acknowledgements:} The authors would like to thank professors Carl De Boor, Dariush Ehsani, Sophia Vassiliadou for helpful discussions, and the anonymous referee for his comments.  They especially would like to thank  professors  Xiuxiong Chen and Jianguo Cao for raising the questions on the closed range property for product domains, which arise naturally in many geometric problems.  
In particular, this paper  answers affirmatively  (see $\S$\ref{sec-chencao} below)  on the closed-range property for the product  domain of an annulus and a ball in  $\cx^n$ (which is not pseudoconvex, a question  raised by X. Chen)  and the product of $D\times \cx P^1$ of the unit disc $D$ in $\cx$ and the Riemann sphere  (which is not Stein, a question raised by  J. Cao.)

\section{The $L^2$ Setting for the $\dbar$ problem}\label{sec-l2setting}
\subsection{Spaces of forms on domains}
We recall the definition and notation used in the $L^2$ theory of the $\dbar$-operator.
Let $D$ be a bounded domain in $\cx^n$, and let $\phi$ be a continuous function on $\overline{D}$.
We denote  by $L^2(D)$ the space of square integrable functions on $D$ with respect to weight $\phi$, 
which has the  norm
\[ \norm{f}= \int_D \abs{f}^2 e^{-\phi} dV,\]
where $dV$ is the volume form on  $\cx^n$ induced by the standard hermitian metric.
(Note that we have suppressed $\phi$ from the notation.)

We denote by $L^2_*(D)$ the space of differential forms with coefficients in $L^2(D)$. More generally,
for any space of functions $\mathcal{F}(D)$ on $D$, we will let $\mathcal{F}_*(D)$ denote the space
of forms with coefficients in $\mathcal{F}$. Then $\mathcal{F}_*(D)$ can be thought of as a 
vector space direct sum 
	\begin{equation}\label{eq-fstard}
 		\mathcal{F}_*(D)= \bigoplus_{\substack{0\leq p\leq n\\0\leq q\leq n}}\mathcal{F}_{p,q}(D)
	\end{equation}
of the spaces of forms of bidegree $(p,q)$.

Often the space $\mathcal{F}(D)$ will be a Hilbert space. Then we can give
$\mathcal{F}_*(D)$ a Hilbert space structure in the following way: first, we declare that forms of different
bidegrees are orthogonal, so that the sum in \eqref{eq-fstard} is now an orthogonal direct sum of Hilbert Subspaces.
Any form $f\in\mathcal{F}_{p,q}(D)$ can be uniquely represented as
\[ f = \sum'_{I,J}{f_{I,J}}dz^I\wedge d\overline{z}^J,\]
where $I=(i_1,\dots,i_p)\in \mathbb{N}^p$, and $J=(j_1,\dots,j_q)\in \mathbb{N}^q$ are multi-indices,
$f_{I,J}\in\mathcal{F}(D)$,
$dz^I=dz_{i_1}\wedge\dots\wedge dz_{i_p}$ and $d\bar{z}^J=d\bar{z}_{j_1}\wedge\dots\wedge d\bar{z}_{j_q}$,
and the notation $\sum'$ means that the summation is over strictly increasing multi-indices only,
i.e. $i_1<i_2<\dots<i_p$ and $j_1<j_2<\dots<j_q$.
We define the norm of $f$ as
\begin{equation}\label{eq-formhilbert}
\norm{f}_{\mathcal{F}_*(D)}^2 = \sum'_{I,J}\norm{f_{I,J}}_{\mathcal{F}(D)}^2.\end{equation}	
In this paper, the Hilbert space $\mathcal{F}$ will be either a usual $L^2$ space (possibly with weight),
a Sobolev space, or a partial Sobolev space on product
domains (to be defined in $\S$\ref{sec-partialsobolev}.)
These notions easily extend to spaces of forms on domains in 
hermitian manifolds (see \cite[Chapter~5]{cs}) using the natural pointwise inner-product on
forms induced by the hermitian structure.

\subsection{The $L^2$ Dolbeault Complex}\label{sec-l2defns}
We now recall the definition of the the $\dbar$-operator on the Hilbert space $L^2_*(D)$
of  forms with square integrable coefficients on $D$.
The $\dbar$-operator is the closed, densely defined unbounded operator from $L^2_*(D)$ to itself which
coincides with the usual $\dbar$ operator from $\smooth_*(\overline{D})$ to
 $\smooth_*(\overline{D})$, and which has been extended as a distributional
 operator to the dense domain of definition
 \[ \dm(\dbar) = \{ f\in L^2_*(D)\colon \dbar f \in L^2_*(D)\}.\]
In the terminology of \cite{brun}, the operator $\dbar$ is the differential map 
of a {\em Hilbert Complex}, i.e., a cochain complex, in which the cochain space
$\dm(\dbar)$ is a dense subspace of a graded Hilbert Space, and the differential is
a closed, densely defined unbounded linear map of the graded Hilbert space into itself.
Note that the map $\dbar$ has bidegree $(0,1)$, i.e., it maps $(p,q)$ forms to $(p,q+1)$ forms.

We denote by $\dbar^*$ the Hilbert space adjoint of $\dbar$. This is again a closed,
densely defined operator on $L^2_*(\Omega)$. Its domain $\dm(\dbar^*)$ is 
in general very different from  $\dm(\dbar)$, because of the natural boundary conditions
that the Hilbert space adjoint must satisfy. The map $\dbar^*$ is of bidegree $(0,-1)$

A form $f\in L^2_*(D)$ is said to be {\em harmonic}, if $\dbar f= \dbar^* f=0$.
The harmonic forms $\mathcal{H}_*(D)$ form a closed subspace of $L^2_*(D)$.
The  orthogonal projection $P:L^2_*(D)\rightarrow \mathcal{H}_*(D)$ is called the 
{\em harmonic projection}, which is of course a map of bidegree $(0,0)$.
Note that since $\dbar^*$ vanishes on $L^2_{0,0}(D)\equiv L^2(D)$, the space
$\mathcal{H}_{0,0}(D)$ can be identified with the space $L^2(D)\cap \mathcal{O}(D)$ 
of square integrable holomorphic functions, the {\em Bergman space} associated to $D$.
The operator $P_{0,0}$ is the {\em Bergman projection} onto square integrable holomorphic 
functions.

\subsection{The closed range property and its consequences}\label{sec-canonicalsolution}
Let $g\in L^2_*(D)$ be such that
$\dbar g=0$. In order to solve the equation $\dbar u=g$ in the $L^2$ sense
first we need to show that the $L^2$ $\dbar$-operator has closed range.
In general, the closed-range property is not easy to establish, even with 
a smooth boundary. Subtle holomorphically invariant convexity properties of
the boundary  of $D$ control whether $\dbar$ has closed range on $D$ 
(see the example on p. 76 of \cite{fk}.) Note that, in contrast, for the
$L^2$ $d$-complex on a Riemannian manifold the operator $d$ always has closed
range when the boundary is $\mathcal{C}^2$ or even Lipschitz
\cite{mcs1, mmt}.

An important consequence of the closed range property on $D$
is the existence of the {\em Canonical-}
or {\em Kohn's solution operator} $K$, which is a bounded map 
from $L^2_*(D)$ to itself of bidegree $(0,-1)$, and
is a right-inverse of the operator $\dbar$. For every
$f\in \im(\dbar)$, we define $Kf$ to be the unique solution
to $\dbar u=f$ which is orthogonal to $\ker(\dbar)$. We then
extend $K$ to all of $L^2_*(D)$ by setting $K\equiv 0$ 
on $(\im (\dbar))^\perp$ and extending linearly.  The map $K$ is bounded
by the closed graph theorem, and is represented in terms
of the $\dbar$-Neumann operator $\mathsf{N}$ on $D$ as $K=\dbar^*\mathsf{N}$.
We  further have the following:
	\begin{lem}\label{lem-homotopy}  If $\dbar$ has closed range, and 
		$K$ is the canonical solution operator, then on $\dm(\dbar)$ we have
			\begin{equation}\label{eq-homotopy}
				I-P=\dbar K+K\dbar.\end{equation}
		Further, the ranges of the three operators $\dbar K$, $K\dbar$ and $P$ are orthogonal.
	\end{lem}
 	\begin{proof}
 		Since $\im(\dbar)$ is closed in $L^2_*(D)$, we have the Strong Hodge decomposition:
		\[ L^2_*(D)=(\im (\dbar^*))\oplus (\im (\dbar))\oplus\mathcal{H}_*(D),\]
		where $\mathcal{H}_*(D)$ is the Hilbert space of harmonic forms,
		and $\oplus$ means that the summands are orthogonal (see \cite{cs}).
		Let $Q$ and $R$ be the orthogonal projections from
		$L^2_*(D)$ onto the closed subspaces $\im(\dbar)$ and
		$\im(\dbar^*)$ respectively. Then on $L^2_*(D)$, we have
		$I=P+Q+R$. From the definition of $K$, we have $\dbar K= Q$.
		Also, on $\dm(\dbar)$ we have $K\dbar=R$ by noting
		that the left hand side is the identity on $\left(\ker(\dbar)\right)^\perp=\im(\dbar^*)$ 
		and zero on the orthogonal complement
		$\ker \dbar$. Equation~\eqref{eq-homotopy} now follows. The 
		last statement follows from the method of proof.
	\end{proof}

For any domain $D$, the {\em $L^2$  Dolbeault Cohomology} space is
the graded vector space
\[ {H}^*_{L^2}(D) = \frac{\ker(\dbar)}{\im(\dbar)}.\]
Note that in the quotient topology, this is a Hilbert space if $\im(\dbar)$
is closed (and not even Hausdorff if $\im(\dbar)$ is not closed, see \cite[Chapter~1, $\S$2.3]{scha}.)
If $\im(\dbar)$ is closed,  we have
\begin{align*}
H^*_{L^2}(D)&\cong(\ker(\dbar))\cap \left(\im(\dbar)\right)^\perp\\
&= (\ker(\dbar))\cap(\ker(\dbar^*))\\
&=\mathcal{H}_*(D),
\end{align*}
so that the cohomology space is naturally isomorphic to the space of 
harmonic forms. We will now recall the less well-known converse to this statement, due to Kodaira (see \cite[p. 165]{dRh}. Let
\[ [.]: \ker(\dbar)\rightarrow H^*_{L^2}(D)\]
denote the natural projection onto the quotient space. We have the following:

\begin{lem}\label{lem-j} Let $\eta$ be the linear map from the vector
space of harmonic forms
$\mathcal{H}_*(D)$ to the cohomology  vector space $H^*_{L^2}(D)$
 given by $ \eta(f)= [f]$. Then
\begin{itemize}
 \item[(i)] $\eta$ is injective.
 \item[(ii)] If $\eta$ is also surjective, then the range of $\dbar$ is closed.
 \end{itemize}
 \end{lem}
\begin{proof}(i) For $(0,0)$-forms, i.e. functions, the space $\mathcal{H}_{0,0}(D)$ coincides by definition
with the cohomology space $H_{L^2}^{0,0}(D)$. For forms of higher degree, a harmonic form in $\ker
(\eta)$ is of the form $\dbar g$ with
$\dbar^*(\dbar g)=0$ so that
\begin{align*}
0&=(\dbar^*(\dbar g),g)\\
&=\norm{{\dbar} g}^2.
\end{align*}

(ii) Since $\eta$ is an isomorphism,  we can identify
$\mathcal{H}_*(D)$ with the cohomology space
$H^*_{L^2}(D)$. Since $\mathcal{H}_*(D)$ is a
closed subspace of the Hilbert Space $L^2_*(D)$, the space
$H^*_{L^2}(D)$ becomes a Hilbert space in the natural way.
Then the map $[\cdot]$ can be thought of as an operator from the
Hilbert space $\ker(\dbar)\subset L^2_*(D)$ to the Hilbert space
$H^*_{L^2}(\Omega)$. Since $\eta$ is surjective, every
element of $\ker(\dbar)$ can be written as $f+{\dbar} g$, where
$f\in \mathcal{H}^*(D)$. Then $[(f+\dbar g)]=f$, using the identification
of $\mathcal{H}_*(D)$ and $H^*_{L^2}(D)$. Since $\norm{f+{\dbar} g}^2=
\norm{f}^2+\norm{{\dbar} g}^2\geq \norm{f}^2$, so that 
$\norm{[f+\dbar g]}\leq \norm{f+\dbar g}$, it follows that
$[\cdot]$ is actually a bounded map. Therefore, $\ker[\cdot]=
\im(\dbar)$ is closed.
\end{proof}

\section{Differential forms on product domains}\label{sec-diffforms}
\subsection{Algebraic tensor product of spaces of forms}
Let $\mathsf{H}_1$ and $\mathsf{H}_2$ be $\cx$-vector spaces.  We denote by
$\mathsf{H}_1\tensor \mathsf{H}_2$  the {\em algebraic} tensor
product (over $\cx$) of  $\mathsf{H}_1$ and $\mathsf{H}_2$ : then $\mathsf{H}_1\tensor \mathsf{H}_2$ can be thought of as
the space of finite sums of elements of the type $x\tensor y$,
where $x\in \mathsf{H}_1$ and $y\in \mathsf{H}_2$,  where
 $\tensor: \mathsf{H}_1\times \mathsf{H}_2\rightarrow \mathsf{H}_1\tensor \mathsf{H}_2$ is the
 canonical bilinear map (see e.g.
\cite[$\S$3.4]{weid} for the purely algebraic definition.)  Similarly $\mathsf{H}=\mathsf{H}_1\tensor \mathsf{H}_2\tensor\dots\tensor \mathsf{H}_N$
denotes the algebraic tensor product of $N$ vector spaces $\mathsf{H}_1,\dots \mathsf{H}_N$. We call an element of $\mathsf{H}$ 
of the form $x_1\tensor\dots\tensor x_N$ a {\em simple tensor.}

When $\mathsf{H}_1,\dots, \mathsf{H}_N$ are realized as spaces of forms on domains (or manifolds)
$\Omega_1,\dots, \Omega_N$,
there is a concrete realization of the algebraic tensor product $\mathsf{H}$ as a space of 
forms on the product domain $\Omega=\Omega_1\times\dots\times\Omega_N$. 
For $j=1,\dots, N$, let $f_j\in \mathsf{H}_j$, so that $f_j$ is a form on the domain $\Omega_j$ and let
$\pi_j:\Omega\rightarrow \Omega_j$ denote the projection onto the $j$-th factor $\Omega_j$ 
from the product $\Omega$. We define a  form on $\Omega$, the {\em tensor
product} of the forms $f_1,\dots, f_N$, by setting
\begin{equation}\label{eq-cross} 
f_1\tensor\dots\tensor f_N = \pi_1^*f_1\wedge\dots\wedge \pi_N^* f_N,
\end{equation}
which we will call a {\em simple decomposable} form. Then $\mathsf{H}_1\tensor\dots\tensor\mathsf{H}_N$
is the linear span of  the simple decomposable forms. It is easy to verify that this construction
gives rise to a vector space isomorphic to the usual algebraic definition of a tensor product by
the universal property.

\subsection{Hilbert tensor products}\label{sec-hilten} We now specialize 
to the   case where the factors $\mathsf{H}_j$ are Hilbert spaces. For ease
of exposition, we assume that $N=2$, and the general case should be obvious.
We can define an inner product on the algebraic tensor product $\mathsf{H}_1\tensor \mathsf{H}_2$
defined above by setting
\[ (x\tensor y, z\tensor w)= (x,  z)_{\mathsf{H}_1}(y, w)_{\mathsf{H}_2},\]
and extending bilinearly. This is well-defined thanks to the bilinearity of $\tensor$.
This makes $\mathsf{H}_1\tensor \mathsf{H}_2$ into a pre-Hilbert space,
and its completion is a Hilbert space denoted by $\mathsf{H}_1\csor \mathsf{H}_2$, the {\em Hilbert tensor product}
of the spaces $\mathsf{H}_1$ and $\mathsf{H}_2$. The algebraic tensor product $\mathsf{H}_1\tensor \mathsf{H}_2$ sits
inside $\mathsf{H}_1\csor \mathsf{H}_2$ as a dense subspace. We will refer to any element of $\mathsf{H}_1\tensor \mathsf{H}_2$ (thought of
as a subspace of $\mathsf{H}_1\csor \mathsf{H}_2$)   as a {\em decomposable
form.} For further details on Hilbert tensor products, see  \cite[$\S$3.4]{weid}, or for  a more intrinsic approach  \cite[$\S$2.6, vol. 1]{kr}.

Now let $\mathcal{F}(\Omega_1)$ and $\mathcal{G}(\Omega_2)$ be Hilbert Spaces of functions
on $\Omega_1$ and $\Omega_2$ respectively, and let $\mathcal{F}_*(\Omega_1)$ and $\mathcal{G}_*(\Omega_2)$
be the Hilbert Spaces of forms with coefficients in $\mathcal{F}(\Omega_1)$ and $\mathcal{G}(\Omega_2)$
respectively, with the norm given by \eqref{eq-formhilbert}. It is easily verified
from the definitions that there is an isometric equality of Hilbert spaces:
\begin{equation}\label{eq-formhilbten} \mathcal{F}_*(\Omega_1)\csor\mathcal{G}_*(\Omega_2)
=(\mathcal{F}\csor\mathcal{G})_*(\Omega_1\times \Omega_2),\end{equation}
with an obvious extension to the case of more than two factor domains.
\subsection{Forms with square integrable coefficients}
We now recall the most important case of the above constructions. Another example will
be considered in $\S$\ref{sec-partialsobolev}.

Recall the following classical fact,  which we will  use repeatedly:
\begin{lem}[{\cite[p.~369]{hor}}]\label{lem-cartan} {\rm Let $\Omega_1, \Omega_2$ be domains 
in Euclidean spaces (or manifolds), and let $\Omega=\Omega_1\times \Omega_2$. Then 
every function in $\smooth_0(\Omega)$ can be approximated
in the $\mathcal{C}^k$ norm (where $0\leq k\leq \infty$) by functions in the algebraic
tensor product $\smooth_0(\Omega_1)\tensor \smooth_0(\Omega_2)$.} 
\end{lem}

Now, $\smooth_0(\Omega)$ is dense in $L^2(\Omega)$, and the decomposable compactly supported
smooth functions $\smooth_0(\Omega_1)\tensor \smooth_0(\Omega_2)$ are
dense in the uniform norm (and therefore  in the $L^2$ norm) in the space 
$\smooth_0(\Omega)$. It follows that:
\[ L^2(\Omega_1)\csor L^2(\Omega_2)= L^2(\Omega_1\times \Omega_2).\]
Combining with \eqref{eq-formhilbten} we have:
\begin{equation}\label{eq-l2tensor}
L^2_*(\Omega_1)\csor L^2_*(\Omega_2)= L^2_*(\Omega_1\times\Omega_2).
\end{equation}

\subsection{Tensor products of operators}\label{sec-optensor} Again, for clarity we confine
ourselves to the case $N=2$. Let $\mathsf{H}_1,\mathsf{H}_2, \mathsf{H}_1',\mathsf{H}_2'$ be Hilbert Spaces.
Given bounded linear operators  $T_1: \mathsf{H}_1\rightarrow \mathsf{H}_1'$ and $T_2: \mathsf{H}_2\rightarrow \mathsf{H}_2'$,
we can define an algebraic tensor product $T_1\tensor T_2$ which maps the algebraic
tensor product $\mathsf{H}_1\tensor \mathsf{H}_2$ into $\mathsf{H}_1'\tensor \mathsf{H}_2'$ : on decomposable tensors it is given
by  $(T_1\tensor T_2)(x\tensor y)= T_1x \tensor T_2 y$ and extended linearly.
Then $T_1\tensor T_2$ is bounded on  the dense subspace $\mathsf{H}_1\tensor \mathsf{H}_2$ and therefore extends
to a bounded linear operator $T_1\csor T_2$ from $\mathsf{H}_1\csor \mathsf{H}_2$ to $\mathsf{H}_1'\csor \mathsf{H}_2'$.

This construction can be extended to densely defined unbounded linear operators, provided
they are {\em closed.} (see \cite[$\S$11.2, vol. 2]{kr}.) Given closed (or even closable)
operators $T_1: \dm(T_1)\rightarrow \mathsf{H}_1'$ and $T_1: \dm(T_2)\rightarrow \mathsf{H}_2'$, where
$\dm(T_1)$ and $\dm(T_2)$ are dense subspaces of the Hilbert spaces $\mathsf{H}_1$ and $\mathsf{H}_2$,
the algebraic tensor product (which is densely defined on $\mathsf{H}_1\csor \mathsf{H}_2$ with domain
$\dm(T_1)\tensor \dm(T_2)$) is  closable  (see \cite[Proposition~11.2.7 (vol. 2)]{kr}.)
Its closure, denoted by $T_1\csor T_2$ is a closed densely defined operator from $\mathsf{H}_1\csor \mathsf{H}_2$
to $\mathsf{H}_1'\csor \mathsf{H}_2'$. Note that this definition agrees with the previous one, when both $T_1$ and
$T_2$ are bounded.

\section{$\dbar$ on a product domain in the  $L^2$ sense} 
\label{sec-dbarl2}
In this section we construct a solution operator to the $\dbar$-problem on 
a product domain in terms of the canonical solution operators on the factor domains,
and show that the operator constructed in fact gives the canonical solution
on  $\ker\dbar_{p,1}$, the $\dbar$-closed $(p,1)$-forms. We use the following notation: for $j=1,\dots,N$, let
$\Omega_j$  be a bounded domain in Euclidean space $\cx^{n_j}$.
All our arguments and results will have easy 
generalizations to relatively compact domains in 
hermitian manifolds, which we leave for the reader.
We will assume that
the boundary of each domain $\Omega_j$ is Lipschitz, i.e.,
it can be represented locally in holomorphic coordinates
as the graph of a Lipschitz  function.  For each $j$, we also fix 
a weight function $\phi_j$ continuous on $\overline{\Omega_j}$. We
use the $L^2$ space  forms $L^2_*(\Omega_j)$ on the domain $\Omega_j$ with 
the weight $\phi_j$, i.e., the norm of a function $f$
is given by $\norm{f}^2=\int_{\Omega_j}\abs{f}^2 e^{-\phi_j} dV$,
where $dV$ is the volume form induced by the hermitian metric. (If we want spaces
without weights, we simply take $\phi_j\equiv 0$.)

The product domain $\Omega=\Omega_1\times\cdots\times\Omega_N$  also 
has Lipschitz boundary. We will consider $L^2_*(\Omega)$ with the 
weight $\phi=\phi_1+\dots +\phi_N$ (and with the product hermitian metric.)
The analog of formula \eqref{eq-l2tensor} holds with this choice of 
metric and weight:
\[ L^2_*(\Omega)= L^2_*(\Omega_1)\csor\dots\csor L^2_*(\Omega_N).\]

Our fundamental assumption will be the following:
{\em For each $j$, the $L^2$ $\dbar$-operator (with  weight $\phi_j$)
on  $\Omega_j$ has closed range in each degree.}

We remark that the closed range property is independent of the weight function $\phi_j$ as 
long as it is continuous to the boundary since the $L^2$ spaces are the same. 
We will show that the $\dbar$ operator on  $\Omega$ (with weight $\phi$)
also has closed range and deduce a formula for the canonical solution on $\ker(\dbar)$.

\subsection{ Construction of solution operator on smooth decomposable forms}
\label{sec-kdecomposable}
For simplicity of exposition, we from now on consider the case $N=2$, that
is we have two domains $\Omega_1$ and $\Omega_2$ and we are trying 
to solve the $L^2$ $\dbar$-problem on the product $\Omega=\Omega_1\times\Omega_2$.
In this section we write down some algebraic formulas which hold for smooth decomposable
forms on $\Omega$.

We first note that if $f\in\smooth_*(\overline{\Omega_1})$ and $g\in\smooth_*(\overline{\Omega_2})$,
then we have

	\begin{equation}\label{eq-leib0}
		\dbar(f\tensor g) = \dbar_1 f\tensor g +\sigma_1 f\tensor \dbar_2 g,
	\end{equation}
where $\dbar_1,\dbar_2,\dbar$ denote the $\dbar$ operator on the domains $\Omega_1,\Omega_2,\Omega$ 
respectively, and $\sigma_1$ is the map on $\smooth_*(\overline{\Omega_1})$ which is multiplication
by $(-1)^{p+q}$ on $\smooth_{p,q}(\overline{\Omega_1})$.
Note that if $T$ be any linear map of odd degree on the space  $\smooth_*(\overline{\Omega_1})$ 
(i.e. the degrees of $Tf$ and $f$
differ by an odd integer) then we obviously have
\begin{equation}\label{eq-sigma}
\sigma_1 T= -T\sigma_1.
\end{equation} 
Extending \eqref{eq-leib0} bilinearly to $\smooth_*(\overline{\Omega_1})\tensor\smooth_*(\overline{\Omega_2})$, 
we obtain the {\em Leibnitz formula} for smooth decomposable forms:
	\begin{equation}\label{eq-leibniz} 
		\dbar=\dbar_1\tensor I_2+ \sigma_1\tensor \dbar_2.
	\end{equation}

Let $K_1,K_2$ be the canonical solution operators on $\Omega_1,\Omega_2$ (see $\S$\ref{sec-canonicalsolution}.)
We define an operator $S$ from $\smooth_*(\overline{\Omega_1})\tensor\smooth_*(\overline{\Omega_2})$
into $L^2_*(\Omega_1)\tensor L^2_*(\Omega_2)$ by the formula
	\begin{equation}\label{eq-kdef}
		S= K_1\tensor I_2 +\sigma_1 P_1\tensor K_2,
	\end{equation} 
where $P_j$ denotes the harmonic projection on the domain $\Omega_j$ (see $\S$\ref{sec-l2defns}.) It will be
proved in the next section that $S$ extends to $L^2_*(\Omega)$, and coincides on 
$(0,1)$-forms with the canonical solution operator on the product $\Omega$. In this section, we take a first step 
in this direction by proving the following homotopy formula:
	\begin{lem}\label{lem-khomotopy}
		On the space of smooth decomposable forms 
		$\smooth_*(\overline{\Omega_1})\tensor\smooth_*(\overline{\Omega_2})$,
		we have
			\begin{equation}\label{eq-khomotopy}
				\dbar S+ S\dbar = I-P_1\tensor P_2,
			\end{equation} 	
		where $I$ is the identity map. 
	\end{lem}
	\begin{proof}
		 First note that
			\begin{align*}
				\dbar S &= (\dbar_1\tensor I_2+\sigma_1\tensor\dbar_2)
					(K_1\tensor I_2 +\sigma_1 P_1\tensor K_2)\\
				&=\dbar_1 K_1\tensor I_2 + \sigma_1 K_1\tensor \dbar_2+ P_1\tensor \dbar_2K_2,
			\end{align*}
		where one term vanishes because $\dbar_1P_1=0$. Similarly, since by the Hodge decomposition,
		 $P_1\dbar_1=0$, we have,
			\begin{align*}
				S\dbar&=(K_1\tensor I_2 + \sigma_1 P_1\tensor K_2)
					(\dbar_1\tensor I_2+\sigma_1\tensor\dbar_2)\\
				&=K_1\dbar_1\tensor I_2 +K_1\sigma_1\tensor \dbar_2+ P_1\tensor K_2\dbar_2\\
				&=K_1\dbar_1\tensor I_2 -\sigma_1K_1\tensor \dbar_2+ P_1\tensor K_2\dbar_2,
			\end{align*}
		where we have used \eqref{eq-sigma} in the last line along with the fact that  $K_1$ has degree $-1$.
		Combining  the  two expressions and canceling the middle terms we have
			\begin{align*}
				\dbar S + S\dbar &= (\dbar_1K_1+K_1\dbar_1)
					\tensor I_2 + P_1\tensor(\dbar_2 K_2 + K_2 \dbar_2)\\
					&= (I_1-P_1)\tensor I_2 +P_1\tensor (I_2-P_2)\\
					&= I_1\tensor I_2- P_1\tensor P_2,
 				\end{align*}
		where we have used the homotopy formula \eqref{eq-homotopy}  in each factor. The result follows.
	\end{proof}

\subsection{Density results: Extension  to $\boldmath{\dm(\dbar)}$}
In this section we use a density argument to extend the 
formulas of the last section. 
We first recall the following:
\begin{lem}[{\cite[Lemma~4.3.2, part~(i)]{cs}}]
\label{lem-fried}{\rm If $D$ is a Lipschitz domain, then the space $\smooth_*(\overline{D})$
of forms with $\smooth(\overline{D})$ coefficients 
 is dense in the graph-norm  in the domain $\dm(\dbar)$  of the $L^2$ $\dbar$ operator on $D$.}
\end{lem}
 Since $D$ is Lipschitz, it is locally star-shaped.  This is a special case
of Friedrichs' Lemma and follows from smoothing by convolution with a mollifier;
see Section 1.2 in Chapter I in H\"ormander \cite{Hor1}
or  Part (i) of proof of  the Density Lemma 4.3.2 in \cite{cs}. The following is now easy:

\begin{lem}\label{lem-density}
$\smooth(\overline{\Omega_1})\tensor\smooth(\overline{\Omega_2})$ 
is  dense in the domain of $\dbar$ in the graph norm of the $\dbar$-operator on
$\Omega=\Omega_1\times\Omega_2$.
\end{lem}
\begin{proof}
Given a form $f\in\dm(\dbar)$ on $\Omega$, by the Lemma~\ref{lem-fried}, we can approximate it
in the graph norm by a form $\tilde{f}\in\smooth_*(\overline{\Omega})$.
Note that it easily follows from Lemma~\ref{lem-cartan} that
every form in $\smooth_*(\overline{\Omega})$ can be approximated
in the $\mathcal{C}^k$ norm (where $0\leq k\leq \infty$) by forms in the algebraic
tensor product 
$\smooth_*(\overline{\Omega_1})\tensor \smooth_*(\overline{\Omega_2})$. 
Therefore, approximating $\tilde{f}$ by a form in $\smooth_*(\overline{\Omega_1})\tensor \smooth_*(\overline{\Omega_2})$
in the $\mathcal{C}^1$ norm (which dominates the graph norm)
our result follows.
\end{proof}

We now extend the formulas of the previous section from the space $\smooth_*(\overline{\Omega_1})\tensor\smooth_*(\overline{\Omega_2})$ of smooth decomposable
forms (which is dense in the graph norm  of $\dbar$)
to $\dm(\dbar)$.
\begin{lem}\label{lem-ext}
On the dense subspace $\dm(\dbar)\subset L^2_*(\Omega)$ we have :
\begin{equation}\label{eq-leibniz-csor} 
		\dbar=\dbar_1\csor I_2+ \sigma_1\csor \dbar_2.
\end{equation}
The operator $S$ defined in \eqref{eq-kdef} can be extended to $L^2_*(\Omega)$ by the formula
\begin{equation}\label{eq-kdef-csor}
		S= K_1\csor I_2 +\sigma_1 P_1\csor K_2,
\end{equation} 
and on $\dm(\dbar)$ the following homotopy formula holds:
\begin{equation}\label{eq-khomotopy-csor}
				\dbar S+ S\dbar = I-P_1\csor P_2.
\end{equation} 	
\end{lem}

\begin{proof}
All three formulas follow from the corresponding formulas for decomposable forms 
by taking limits, using  Lemma~\ref{lem-density} for \eqref{eq-leibniz-csor} and
\eqref{eq-khomotopy-csor}.
\end{proof}

\subsection{Consequences}
Using the homotopy formula \eqref{eq-khomotopy-csor}, we can now prove:
\begin{thm} Let $\Omega_1$ and $\Omega_2$ be two bounded  domains in complex hermitian manifolds
with Lipschitz boundaries.
Suppose  that the $\dbar$ operator has closed range in $L^2(\Omega_j)$  for all degrees, where $j=1,2$. Then 
$\dbar$ has closed range in $L^2(\Omega)$  for the product domain $\Omega=\Omega_1\times\Omega_2$.
\end{thm}
\begin{proof}
We recall the result established in Lemma~\ref{lem-j}: if the map $\eta(f)=[f]$
is surjective, then $\dbar$ has closed range. In other words, we need to show that
for every cohomology class $\alpha\in H^*_{L^2}(\Omega)$ there is a {\em harmonic form}
$h\in \mathcal{H}_*(\Omega)$ such that $\alpha=[h]$. We will actually do better. We will
show that there is such a $h$ in the tensor product $\mathcal{H}_*(\Omega_1)\csor\mathcal{H}_*(\Omega_2)\subset\mathcal{H}_*(\Omega)$.
Note that this will also show that 
\begin{equation}\label{eq-harmonictensor}
	\mathcal{H}_*(\Omega_1)\csor\mathcal{H}_*(\Omega_2)=\mathcal{H}_*(\Omega).
\end{equation}

Indeed, let $f\in \ker(\dbar)$ be a form representing the cohomology class $\alpha$, 
i.e. $\alpha=[f]$. Then, from the homotopy formula \eqref{eq-khomotopy-csor}, we have
\[ f- \dbar(Kf)= (P_1\csor P_2)f.\]
Therefore, the form $(P_1\csor P_2)f\in \mathcal{H}_*(\Omega_1)\csor\mathcal{H}_*(\Omega_2)$
also represents the same cohomology class $\alpha$, i.e. $[(P_1\csor P_2)f]=\alpha$. Therefore
every cohomology class in $H^*_{L^2}(\Omega)$ can be represented by a harmonic form in 
$\mathcal{H}_*(\Omega)$ (indeed by a harmonic form in the possibly smaller subspace 
$\mathcal{H}_*(\Omega_1)\csor\mathcal{H}_*(\Omega_2)$.) This shows that the map $\eta$ of
Lemma~\ref{lem-j} is surjective. The equality \eqref{eq-harmonictensor} now follows from the
fact that $\eta$ is injective.
\end{proof}
We now note a few important consequences of the above result:
\begin{cor}(i) The $L^2$ K\"unneth formula holds for the Dolbeault cohomology with $L^2$ coefficients:
\begin{equation}\label{eq-kunneth}
H^*_{L^2}(\Omega)=H^*_{L^2}(\Omega_1)\csor H_{L^2}^*(\Omega_2)
\end{equation}

(ii) The harmonic projections satisfy $P=P_1\csor P_2$
\end{cor}
\begin{proof} Part (i) follows from the natural isomorphisms 
$H^*_{L^2}(\Omega)\cong\mathcal{H}_*(\Omega)$,  $H^*_{L^2}(\Omega_1)\cong\mathcal{H}_*(\Omega_1)$
and  $H^*_{L^2}(\Omega_2)\cong\mathcal{H}_*(\Omega_2)$ (note that the range of $\dbar$ is 
closed in each case.) Part (ii) follows from comparing the homotopy 
formulas \eqref{eq-khomotopy-csor} and \eqref{eq-homotopy}, or directly from  \eqref{eq-harmonictensor}.
\end{proof}

We now come to the most significant consequence:
\begin{thm}\label{prop-kiscansol}
For $0\leq p\leq n$, the restriction of the map $S$ defined in \eqref{eq-kdef-csor} to the $\dbar$-closed
$(p,1)$-forms coincides with the restriction of the canonical solution operator $\dbar^*\mathsf{N}$ to the 
same space. 
\end{thm}

\begin{proof} From the Hodge decomposition, we have for $(p,q)$ forms that
$\ker(\dbar_{p,q})= \im(\dbar_{p,q-1})\oplus \mathcal{H}_{p,q}(\Omega)$. If $q=0$,
it follows that 
\begin{align*}
\ker(\dbar_{p,0})&= \mathcal{H}_{p,0}(\Omega)\\
&= \bigoplus_{j+k=p}\mathcal{H}_{j,0}(\Omega_1)\csor \mathcal{H}_{k,0}(\Omega_2),
\end{align*}
by \eqref{eq-harmonictensor}.

We claim that the range of $S_{p,1}$ is orthogonal to the space $\ker(\dbar_{p,0})$.
By the computation above, it is sufficient to show that 
the range of $S_{p,1}$ is orthogonal to every form of the type $g_1\tensor g_2$,
where $g_1$ and $g_2$ are harmonic forms of degrees $(j,0)$ and $(p-j,0)$, where $0\leq j\leq p$.
Let $f_1, f_2$ be $L^2$ forms such that $f_1\tensor f_2$ is of bidegree $(p,1)$. Then,
\begin{align*}
(S(f_1\tensor f_2),g_1\tensor g_2)&= (K_1f_1\tensor f_2+\sigma_1 P_1f_1\tensor K_2f_2,g_1\tensor g_2)\\
&=(K_1f_1,g_1)(f_2,g_2)+(\sigma_1P_1f_1,g_1)(K_2f_2,g_2)\\
&= 0\cdot (f_2,g_2)+(\sigma_1 P_1f_1,g_1)\cdot 0& \\
&=0,
\end{align*}
where  we have used the fact that $K_1,K_2$ being canonical solutions, have ranges orthogonal to $\dbar$-closed forms.

If $f$ is a $\dbar$-closed $(p,1)$ form orthogonal to the harmonic forms, it follows from formula
\eqref{eq-khomotopy-csor} that $\dbar(Sf)=f$. Since $Sf$ is orthogonal to $\ker(\dbar)$, it follows that
$Sf=\dbar^*\mathsf{N}f$.

To complete the proof, we need to show that $S$ vanishes on the space of $(p,1)$ harmonic
forms. By  formula \eqref{eq-harmonictensor}, it follows that we only need to verify this
on a harmonic form of the type $f\tensor g$, where $f,g$ are also harmonic forms. 
We have,
\begin{align*}
S(f\tensor g)&=K_1f\tensor g+\sigma_1P_1f\tensor K_2g\\
&= 0\tensor g+f\tensor 0\\
&=0,
\end{align*}
since $K_1$ and $K_2$ are the canonical solutions on the domains $\Omega_1$ and $\Omega_2$.
\end{proof}
{\em Remark:}  For arbitrary degrees,
the operator $S$ is not equal to the canonical solution operator $K=\dbar^*\mathsf{N}$.
 In fact, an examination of 
the proof of Lemma~\ref{lem-homotopy} shows that for the canonical solution $K$ on 
a domain, the ranges of  the operators $\dbar K$ and $K\dbar$ are orthogonal. On the other hand,
using the computations used in the proof of Lemma~\ref{lem-khomotopy}, we can check that 
\[ \left(\dbar S(f\tensor g), S\dbar(f\tensor g)\right)= - \norm{K_1 f}^2 \norm{\dbar_2 g}^2,\]
so that $S$ is not the canonical solution on the product.

Using a simple induction argument, we can extend the results of this section to
$N$ factors. Further, as remarked above, all the arguments generalize to relatively
compact domains in hermitian manifolds:
\begin{thm}\label{thm-all}For $j=1,\dots, N$, let $M_j$ be a hermitian
manifold and  let $\Omega_j\Subset M_j$ be a Lipschitz domain.    Suppose  that the $L^2$ $\dbar$-operator
on $\Omega_j$ (with weight $\phi_j$)  has closed range for each $1\leq j\leq N$.   Then we have the following:
\begin{itemize}
\item the $\dbar$-operator (with weight $\sum_{j=1}^N\phi_j$) has closed range on $\Omega$.
\item the $L^2$ K\"{u}nneth formula holds:
\[   H^*_{L^2}(\Omega)= H^*_{L^2}(\Omega_1)\csor \dots\ \csor H^*_{L^2}(\Omega_N)\]
\item the  harmonic projection on $\Omega$ is given by
\begin{equation}\label{eq-bergman} P=P_1\csor \dots \csor P_N.\end{equation}
\item a solution operator for $\dbar$   on $\Omega$ is given by 
\begin{equation}\label{eq-tnj}
 S=\sum_{j=0}^{N-1} T_{N,j},\end{equation}
with 
\[ T_{N,j}=\tau_j\mathsf{Q}_j\csor K_{j+1}\csor \mathsf{I}_j,\]
where 
\begin{itemize}
\item $\mathsf{Q}_j$ is the harmonic projection on the domain $U_j=\Omega_1\times\dots\times \Omega_j$,
(the product of the first $j$ factors),
\item  $\tau_j$ is the map on $L^2_*(U_j)$ which multiplies forms of degree $d$ by $(-1)^d$,
\item
$\mathsf{I}_j$ is the identity map on forms on $\Omega_{j+2}\times\dots\times\Omega_{N}$, and
\item
it is understood that $T_{N,0}=K_1\csor \mathsf{I}_0$ and $T_{N,N-1}=\tau_{N-1}\mathsf{Q}_{N-1}\csor K_N$.
\end{itemize}
\item let $0\leq p \leq \sum_{j=1}^N\dim_{\cx} M_j$; on the space of $\dbar$-closed $(p,1)$ forms on $\Omega$, 
the solution operator $S$ coincides with the canonical
solution operator $\dbar^*\mathsf{N}$ of the $\dbar$-equation.
\end{itemize}
\end{thm}
In particular, this proves Theorem~\ref{thm-main}.

\section{Partial Sobolev spaces}
\label{sec-partialsobolev}

\subsection{Definitions}

Recall that for a Lipschitz domain $D$ in $\rl^n$, and an integer $k\geq 0$,
the Sobolev space $W^k(D)$ is the Hilbert space obtained by completion of 
$\smooth(\overline{D})$ under the norm given by
\[ \norm{f}_{W^k(D)}^2 = \sum_{\length{\alpha}\leq k} \norm{\der^\alpha f}_{L^2(D)}^2,\]
where $\alpha=(\alpha_1,\dots,\alpha_n)$ is a multi-index, $\length{\alpha}=\alpha_1+\dots+\alpha_n$
is the length of multi-index, and  $\der^\alpha$ is
the partial derivative operator of order $\alpha$:
\[ \der^\alpha= \frac{\partial^{\length{\alpha}}}{\partial^{\alpha_1} x_1\dots \partial^{\alpha_n}x_n}. \]
We will obtain regularity estimates for the canonical solution on product domains in  a generalized
type of Sobolev space suited to the product structure of the domain. We will call these spaces {\em
partial Sobolev spaces}. Such spaces are characterized by the fact that there are some values 
of the  integer
$l$ such that the norm  controls only {\em some} distinguished partial derivatives of  order $l$. 
For the usual Sobolev space $W^k(D)$, the norm controls either all or no derivatives of
order $l$, depending on whether $l\leq k$ or $l>k$.

For convenience of exposition, first consider a product domain $D\Subset\rl^n$ represented as
$D=D_1\times D_2$, where $D_1\Subset\rl^{n_1}$ and $D_2\Subset\rl^{n_2}$ are Lipschitz domains,
with $n=n_1+n_2$. Let $\alpha=(\alpha_1,\dots,\alpha_n)$ be a multi-index with $n$ components.
We can write $\alpha=\alpha(1)+\alpha(2)$, where
\[ \alpha(1)=(\alpha_1,\dots,\alpha_{n_1},\underbrace{0,\dots,0}_{n_2}),\]
and
\[ \alpha(2)=(\underbrace{0,\dots,0}_{n_1},\alpha_{n_1+1},\dots,\alpha_n).\]
Then $\der^{\alpha(1)}$ acts only on the variables which come from $D_1$ and 
$\der^{\alpha(2)}$ acts only on the variables that come from $D_2$ in the product
$D$, and we have $\der^\alpha=\der^{\alpha(1)}\der^{\alpha(2)}$.

The $\pss^k$-norm of a function $f\in\smooth(\overline{D})$ is defined to be
\begin{equation}\label{eq-pssnorm} \norm{f}_{\pss^k(D)}= \sum_{\substack{\length{\alpha(1)}\leq k\\\length{\alpha(2)}\leq k}}
\norm{\der^\alpha f}^2_{L^2(D)}\end{equation}
Note that the $\pss^k$-norm dominates the ordinary $W^k$-norm on $D$, and is in turn dominated by the 
$W^{2k}$-norm.

We now define the space $\pss^k(D)$ to be the completion of $\smooth(\overline{D})$ under the
norm \eqref{eq-pssnorm}. It is clear how to extend this definition to more than two factors: if 
$D=D_1\times\dots\times D_N$, then the $\pss^k$-norm on $D$ is defined as
 \[ 
\norm{f}_{\pss^k}^2(D) = \sum_{\substack{\length{\alpha(j)}\leq k\\1\leq j\leq N}} \norm{\der^\alpha f}^2_{L^2(D)},
\]
where $\alpha(j)$ is the part of the multi-index $\alpha$ corresponding to the factor $D_j$, defined in analogy 
with the case $N=2$ considered above.
\subsection{Basic Properties}
We now summarize the basic properties of partial Sobolev space $\pss^k(D)$,
where $D=D_1\times\dots\times D_N$.
From the definition, $\pss^k(D)$ is a Hilbert space in the $\pss^k$-norm.
For $k=0$, the space $\pss^0(D)$ coincides with $L^2(\Omega)$. In general, for each $k$, we have continuous inclusions:
\begin{equation}\label{eq-inclusion}
\mathcal{C}^{Nk}(\overline{D})\hookrightarrow 
W^{Nk}(D)\hookrightarrow \pss^k(D)\hookrightarrow W^k(D)\hookrightarrow L^2(D).
\end{equation}
Since $\bigcap_{k\geq 0} W^{Nk}(D)=\bigcap_{k\geq 0} W^k(D)=\smooth(\overline{D})$, it follows that
\begin{equation}\label{eq-cinfty} \bigcap_{k\geq 0} \pss^k(D)=\smooth(\overline{D}).\end{equation}

The significance of these spaces is explained by:
\begin{lem}
For $j=1,\dots,N$, let $\Omega_j\Subset\cx^{n_j}$ be a Lipschitz domain, and denote
  the product by $\Omega=\Omega_1\times\dots\times\Omega_N$. Then we have an
isometric equality of Hilbert spaces of forms on $\Omega$:
\begin{equation}\label{eq-psstensor}\pss^k_*(\Omega) = W^k_*(\Omega_1)\csor\dots\csor W^k_*(\Omega_N).\end{equation}
\end{lem}
\begin{proof} For simplicity of exposition, we assume $N=2$. Thanks to the 
comments in $\S$\ref{sec-hilten}, in particular equation \eqref{eq-formhilbten},
it follows that we only need to show that
\[ \pss^k(\Omega) = W^k(\Omega_1)\csor W^k(\Omega_2).\]

Thanks to Lemma~\ref{lem-cartan}, it follows easily by $\mathcal{C}^{2k}$ approximation,
that $\smooth(\overline{\Omega_1})\tensor\smooth(\overline{\Omega_2})$ is
dense on each side. Therefore, all it needs to prove isometric equality is 
to show that the $\pss^k$ norm and the tensor product norm coincide on this subspace.
A computation shows that  we have 
$\norm{f\tensor g}_{\pss^k(\Omega_1\times \Omega_2)}= \norm{f}_{W^k(\Omega_1)}\norm{g}_{ W^k(\Omega_2)}=\norm{f\tensor g}_{W^k(\Omega_1)\csor W^k(\Omega_2)}$
\end{proof}
\subsection{Partial Sobolev Spaces on Manifolds} When for each $j$, the domain  $\Omega_j$ is smoothly bounded
in a hermitian manifold
$M_j$, we can again define the partial Sobolev space $\pss^k(\Omega)$ on the product. The simplest
approach is to take \eqref{eq-psstensor} to be the definition and deduce the 
description in terms of distinguished derivatives from there. Alternatively,
one can use a partition of unity to define $\pss^k(\Omega)$ subordinate to a covering 
of $\overline{\Omega}$ by coordinate patches. 

\section{Regularity Results}\label{sec-regularity}
We now prove some results regarding the regularity of the
solution of the $\dbar$-equation on product domains.
Our main tool is the operator $S$ defined in $\S$\ref{sec-dbarl2}.
\subsection{Proof of Theorem~\ref{thm-reg}}
By Theorem~\ref{prop-kiscansol}, the solution operator $S$ on 
the product $\Omega$ coincides with the canonical solution operator on
$\dbar$-closed $(p,1)$-forms. Therefore, it is sufficient to show that $S$ 
is bounded from $\pss^l_{p,1}(\Omega)$ to itself. In fact, it is easy
to see that $S$ is bounded from $\pss^l_*(\Omega)$ to itself.

The regularity of the $\dbar$-Neumann operator on $W^k(\Omega_j)$ for each $k\geq 0$
implies that the canonical solution operator as well as the harmonic projection
preserves the space of forms with $W^k$ coefficients for each $k$ (see \cite[Theorem~6.2.2 and Theorem~6.1.4]{cs};
note that in this reference (i) the hypothesis of pseudoconvexity is used only to deduce that the $\dbar$-Neumann
operator is bounded in each Sobolev space, and (ii) although the arguments are stated only for domains in $\cx^n$,
they generalize easily to relatively compact domains in complex manifolds; for similar
results on the Bergman projection, see \cite{BS2}.)
Since $S$ is given by  \eqref{eq-tnj}, in the notation of theorem~\ref{thm-all}, we have
\begin{align*}
T_{N,j}&=\tau_j\mathsf{Q}_j\csor K_{j+1} \csor \mathsf{I}_j\\
&= \tau_j P_1\csor\dots P_j \csor K_{j+1}\csor I_{\Omega_{j+2}}\dots\csor I_{\Omega_N},
\end{align*}
where $P_\nu$ is the harmonic projection, $K_\nu$ is the canonical solution operator and  
$I_{\Omega_\nu}$  is the identity map on $L^2_*(\Omega_\nu)$. Therefore, the $\nu$-th factor in the tensor product
representing $T_{N,j}$ is a bounded linear map on   $W^k_*(\Omega_\nu)$.
It follows (see $\S$\ref{sec-optensor}) that $T_{N,j}$ defines a bounded linear map from the tensor product
$W^k_*(\Omega_1)\csor\dots\csor W^k_*(\Omega_N)$ to itself, i.e., it is a bounded linear map from
$\pss^k_*(\Omega)$ to itself. The solution operator $S$ being the sum of the $T_{N,j}$'s is bounded on $\pss^k_*(\Omega)$.
The proof is complete. 

We note here that the hypothesis  of Theorem~\ref{thm-reg} are not really necessary. All we need to know
to conclude that the canonical solution has coefficients  in $\pss^l(\Omega)$, if the form $f$ has coefficients 
in $\pss^l(\Omega)$ is the following: for each $j$, both the canonical solution and the harmonic projection on each
factor $\Omega_j$ preserves the Sobolev space $W^l(\Omega)$. 

\subsection{Application to products of weakly pseudoconvex domains}
We now consider the $\dbar$-equation on a product of smoothly bounded pseudoconvex
domains:
\begin{cor}\label{cor-main}
For $j=1,...,N$, let $\Omega_j$ be a   bounded pseudoconvex 
domain with smooth boundary in a Euclidean space $\cx^{n_j}$. For $n=n_1+\dots+n_N$,
let $\Omega\subset\cx^n$  be the product domain $\Omega=\Omega_1\times\dots\times\Omega_N$.
Then, for each $k\in \mathbb{N}$,  there is an $C_k>0$ such that, if $t>C_k$,
and we use the weight $\phi_t(z)=t\abs{z}^2$ on $\cx^n$, we have
\begin{itemize}
\item for $1\leq q \leq n$, given a $\dbar$-closed form $f$ in the partial
Sobolev space $\pss^k_{0,q}(\Omega)$, the form $u=Sf$ is in $\pss^k_{0,q-1}(\Omega)$.
The form $u$ satisfies $\dbar{u}=f$,   provided $f$ is orthogonal to the harmonic forms.
\item if $q=1$, further we have that $u$ coincides with $
\dbar^*_t\mathsf{N}_t f$, the canonical solution with 
weight $t$.
\end{itemize}
\end{cor}

\begin{proof}
By the classical solution by Kohn of the weighted $\dbar$-Neumann
problem (see \cite[Theorem~6.1.3]{cs}), for each $j=1,\dots,N$, given an
integer $k\geq 0$, there is a $C_k^j>0$,
such that if $t>C_k^j$, the $\dbar$-Neumann operator is bounded 
on $W^k(\Omega_j)$ provided the weight is taken to be the function
 $\phi^j_t$ on $\cx^{n_j}$ given by $\phi^j_t(z)=t\abs{z}^2$.
The result now follows using the same method as in  Theorem~\ref{thm-reg}, on taking
$C_k= \max_{1\leq j\leq N} C_k^j$ and  noting that $\sum_{j=1}^N\phi^j_t =\phi_t$. 
\end{proof}

Therefore, it is always possible to solve the $\dbar$-equation in a product of pseudoconvex
domains, with estimates in $\pss^k(\Omega)$ using the weight $\phi_t$.
Using the inclusions \eqref{eq-inclusion} and standard results on interpolation, it follows 
that for each $s\geq 0$, and the operator $S$ maps forms with coefficients in $W^s(\Omega)$ to
forms with coefficients in $W^{\frac{s}{N}}(\Omega)$. From this, using a standard ``Mittag-Leffler argument" 
(see \cite[pp. 127ff., {\em Proof of Theorem 6.1.1.}]{cs}),
one can deduce the following from Corollary~\ref{cor-main}: 
\begin{cor}\label{cor-ms4} Under the same assumption as in Corollary~\ref{cor-main},
if $f\in\smooth_{p,q}(\overline{\Omega})$,
is a $\dbar$-closed form, with $q\not=0$,
then there exists  $u\in \smooth_{p,q-1}(\overline{\Omega})$ such that $\dbar u=f$. \end{cor}

For domains which are the intersection of a finite number of smoothly
bounded pseudoconvex domains, such that the boundaries meet
transversely at each point of intersection,
the existence of a solution to the $\dbar$-equation smooth up to the boundary has been
obtained before \cite{ms4} using integral kernels. This includes the result of Corollary~\ref{cor-ms4},
but our method here is simpler and also leads to estimates in Sobolev spaces.
\subsection{Proof of Corollary~\ref{cor-strong}}  Fix $1\leq j\leq N$ and let $0\leq p\leq n_j$.
The canonical  solution operator on $\Omega_j$ maps the space $\smooth_{p,1}(\overline{\Omega})$ of
$(p,1)$-forms smooth up to the boundary to the space $\smooth_{p,0}(\overline{\Omega})$. Using
formula \eqref{eq-homotopy} on $(p,0)$ forms, we see that the harmonic projection $P_j$ 
preserves the space $\smooth_{p,0}(\overline{\Omega_j}).$

By the Sobolev embedding theorem, the Sobolev norms $\norm{\cdot}_{W^k(\Omega_j)}$ form 
a system of seminorms which define the usual Fr\'{e}chet space structure on $\smooth(\overline{\Omega_j})$.
Using a Fr\'{e}chet space version of the closed graph theorem (see e.g. \cite[Theorem~3 on p.~301]{hor}),
we easily see that the map $K_j$ is continuous from $\smooth_{p,1}(\overline{\Omega_j})$ to $\smooth_{p,0}(\overline{\Omega_j})$ and $P_j$  is  continuous from $\smooth_{p,0}(\overline{\Omega_j})$ to itself.
Using  the characterization of continuous
linear maps between Fr\'{e}chet spaces (see \cite[Proposition~2 on p.~97]{hor}), we conclude that for each $l\in\mathbb{N}$,
there is an $k=k(l,j,p)$ such that $K_j$  maps the Sobolev space $W^k_{p,1}(\Omega_j)$ continuously
to the Sobolev space $W^l_{p,0}(\Omega_j)$ and $P_j$ maps the Sobolev space $W^k_{p,0}(\Omega_j)$ to 
the Sobolev space $W^l_{p,0}(\Omega_j)$.
e can assume that for each $l$, the integer
$k_l= k(l,j,p)$ has been chosen to be independent of $j$ and $p$. Also, since $P_j$ is a projection, it 
follows that $k_l\geq l$.

Using the formula \eqref{eq-tnj}, the argument used in the proof 
of Theorem~\ref{thm-reg} shows that  
the operator $S$ 
maps the Partial Sobolev space $\pss^{k_l}_{\boldsymbol{p},1}(\Omega)$ to  $\pss^l_{\boldsymbol{p},0}(\Omega)$
for each integer $l$. It follows from \eqref{eq-cinfty} that $S$ maps $\smooth_{\boldsymbol{p},1}(\overline{\Omega})$
to $\smooth_{\boldsymbol{p},0}(\overline{\Omega})$. Using Theorem~\ref{prop-kiscansol} the 
smoothness up to the boundary of $\dbar^* \mathsf{N}f$ follows whenever $\dbar f=0$ and the $(p,1)$-form 
$f$ is smooth up to the boundary. The statement regarding the Bergman projection now follows from 
the formula $B=I-\dbar^*\mathsf{N}\dbar= I- K\dbar$.

\subsection{Some special product domains}\label{sec-chencao}
We will apply our results to some special cases when the domain
is not pseudoconvex or Stein.  The first case  is the product of
an annulus between two pseudoconvex domain  and a  pseudoconvex domain.

\begin{cor}  Let $\Omega_1=D_2\setminus\overline D_1$ be the annulus between two pseudoconvex domains $D_1\subset
\subset D_2\Subset \cx^n$  with smooth  boundary and let   $\Omega_2$ be a bounded  pseudoconvex domain in $\cx^m$ with Lipschiz boundary. 
Let $\Omega$ be  the  product domain  $\Omega=\Omega_1\times \Omega_2$.  
Then the $\dbar$ operator on $L^2(\Omega)$ has closed range.  Furthermore, for $0\leq p\leq n+m$, we have

\[ \dim H^{p,q}_{L^2}(\Omega)=\dim\mathcal{H}_{p,q}(\Omega)=\begin{cases} \infty, & \text{if $q=0$;}\\
0, & \text{if $q\not=0$ or $q\not=n-1$;}\\
\infty & \text{if $q=n-1$.} \end{cases}
\]
\end{cor}
This corollary follows easily from the fact that $\dbar$ has closed range on  any bounded pseudoconvex domain in $\cx^n$  by H\"ormander \cite{Hor1} (regardless of the smoothness of the boundary) and for the annulus between smooth pseudoconvex domains (see  \cite{mcs1,mcs3}) for all degrees.  It also follows from the H\"ormander's $L^2$ existence theorems, the harmonic space $\mathcal{H}_{p,q}(\Omega_2)$ on the pseudoconvex domain  $\Omega_2$  vanishes unless $p=q=0$, when $\mathcal{H}_{0,0}$ is 
the space of  $L^2$  holomorphic functions.  For the annulus, we have that the cohomology $\mathcal{H}_{p,q}(\Omega_1)$ vanishes except for $q=0$ and $q=n-1$. Thus the corollary follows from the theorem above. 

For   the annulus  between two  concentric balls $\Omega_1=\{z\in\cx^n\colon
1<\abs{z}<2\}$,  the  nontrivial harmonic spaces $\mathcal H_{(p,n-1)}(\Omega_1)$    have  been  computed  explicitly by  H\"ormander (see Theorem 2.2 and equation (2.3)  in  \cite{Hor2}).  
We can  apply the corollary to the   case when  $\Omega=\{z\in\cx^n\colon
1<\abs{z}<2\}\times \{z\in\cx^m\colon \abs{z}<1\}$ in $\cx^{n+m}$.
 In this case, the closed range property for the ball   follows from the work of Kohn \cite{Ko1}.  For the annulus between two balls, the closure 
of the range of $\dbar$ in degree $(p,q)$ follows from \cite[pp. 57 ff.]{fk} for $q\not = n-1$ and
from  \cite{Hor2} if $q=n-1$. Thus $\dbar$ has closed range in the product domain $\Omega$. 
  The harmonic space  $\mathcal H_{(0,0)}$ on $\Omega$  is spanned by the monomials in $\cx^{m+n}$.  The  other nontrivial harmonic spaces $\mathcal H_{(p,n-1)}(\Omega)$   can be expressed   explicitly as the Hilbert tensor products  of harmonic forms  $\mathcal H_{(p,n-1)}(\Omega_1)$  with monomials in $\cx^m$.      We can therefore obtain a complete  description of the harmonic forms  in terms 
of Hilbert tensor products of spaces. Moreover, we have the following existence and regularity  results   for the $\dbar$-operator.

\begin{cor}  Let   $\Omega=\{z\in\cx^n\colon
1<\abs{z}<2\}\times \{z\in\cx^m\colon \abs{z}<1\}=\Omega_1\times\Omega_2\Subset \cx^{n+m}$,  $n\ge 1$ and $m\ge 1$.  
Then the $\dbar$-Neumann operator $\mathsf{N}$ exists on $\Omega$.  For any $(p,q)$-form $f$  with $\pss^k(\Omega)$ (or  $C^\infty(\overline\Omega)$)  coefficients,  where $k$ is any nonnegative integer and $1\leq q\leq n+m$,  such that $\dbar f=0$ and $f\perp   \mathcal{H}_{p,q}$,
there exists a solution $u  $ which has $\pss^k(\Omega)$ (or  $C^\infty(\overline \Omega)$) coefficients  with  $\dbar u=f$ in $\Omega$.
If $q=1$, 
we can choose 
$u=\dbar^*\mathsf{N}f$ to be the canonical solution. 
  \end{cor}

This answers the  question  posed by X. Chen.  
Another interesting case is when one of the factors in the product  is a compact manifold. In this case, the domain is pseudoconvex in the sense of Levi, but not Stein.  Our theorem can also be applied to  the following case. 
  
 \begin{cor}   
Let $\Omega=\Omega_1\times M$ be  the  product  of a bounded  pseudoconvex domain  $\Omega_1$ in $\cx^n$ and let $M$ be a compact complex hermitian manifold. 
Then the $\dbar$ operator on $L^2(\Omega)$ has closed range and the  Harmonic spaces satisfy the K\"{u}nneth formula 
\begin{equation} 
\mathcal{H}_*(\Omega_1)\tensor\mathcal{H}_*(M)=\mathcal{H}_*(\Omega).
\end{equation}
 
\end{cor} 
In this case, the  space $ \mathcal{H}_*(M)$ is finite dimensional and the Hilbert Tensor product
coincides with the algebraic tensor product.  In particular,   $\dbar$ has closed range on 
  the product $\mathbb{D}\times \cx\mathbb{P}^1$ of the disc and the Riemann sphere,
each with its natural metric,  thus answering  a question raised by J. Cao.


\begin{thebibliography}{12}

\bibitem{ba1}   Barrett, David:   
  Regularity of the Bergman projection on domains with
transverse symmetries. {\em Math. Ann.,} {\bf  258}
  (1982)    441--446
 

 

\bibitem{ba2}  Barrett, David: 
  Behavior of the Bergman projection on the
Diederich-Fornaess worm.
 {\em Acta Math} {\bf 168} (1992) 1-10 
 

\bibitem{BV}       Barrett, David  and Vassiliadou, Sophia, K.:  The
Bergman kernel on the intersection of two balls in $\cx^2$. 
{\em Duke Math. J.}
  {\bf 120}  (2003)   441-467
 

\bibitem{jb}Bertrams, Julia:
 Randregularit\"{a}t von L\"{o}sungen der $\overline\partial$ -Gleichung auf dem Polyzylinder und zweidimensionalen analytischen Polyedern.  Dissertation, Rheinische Friedrich-Wilhelms-Universit\"{a}t Bonn, Bonn, 1986.
 
\bibitem{BS2} Boas, Harold P., and Straube, Emil J.:
Equivalence of regularity for the Bergman projection and the $\overline \partial$-Neumann operator.
{\em Manuscripta Math.} {\bf 67} (1990), no. 1, 25--33. 

 \bibitem{BS}   Boas, Harold  P., and Straube, Emil. J.: Sobolev estimates for the
$\overline\partial$-Neumann operator on domains in ${\cx}^n$ admitting a defining function that is
plurisubharmonic on the boundary.
  {\em Math. Zeit.} {\bf 206}  (1991)  81--88

 


\bibitem{brun} Br\"{u}ning; Lesch, M.:
 Hilbert complexes.
 {\em J. Funct. Anal.} {\bf 108}  (1992),  no. 1, 88--132.
 
 \bibitem{Ca}   Catlin, David:
: Subelliptic estimates for the
$\overline\partial$-Neumann problem on pseudoconvex domains.
{\em  Ann. Math.}  {\bf 126}  (1987) 131--191
 
\bibitem{cproc} Chakrabarti, Debraj; Spectrum of the complex Laplacian in product domains.
To appear in  {\em Proc. Amer. Math. Soc. }  Available online at \texttt{arxiv.org}.

\bibitem{cs}
Chen, So-Chin and Shaw, Mei-Chi:
{\em  Partial differential equations in several complex variables.}
AMS/IP Studies in Advanced Mathematics, 19.
American Mathematical Society, Providence, RI;
International Press, Boston, MA, 2001.
\bibitem{dRh}de Rham, Georges.
{\em Vari\'{e}t\'{e}s diff\'{e}rentiables. Formes, courants, formes harmoniques.}  Hermann et Cie, Paris, 1955.
\bibitem{eh1} Ehsani, Dariush:
 Solution of the $\overline\partial$-Neumann problem on a non-smooth domain.
 {\em Indiana Univ. Math. J.}{\bf  52 } (2003),  no. 3, 629--666.
\bibitem{eh2} Ehsani, Dariush:
 Solution of the $\overline\partial$-Neumann problem on a bi-disc.
{\em  Math. Res. Lett.}{\bf   10 } (2003),  no. 4, 523--533.
\bibitem{eh3}Ehsani, Dariush:
 The $\overline\partial$-Neumann problem on product domains in $\cx^n$.
{\em  Math. Ann.}  {\bf 337}  (2007)  797--816.

\bibitem{fu}Fu, Siqi:
 Spectrum of the $\dbar$-Neumann Laplacian on polydiscs.
{\em  Proc. Amer. Math. Society}  {\bf  135}  (2007) 725-730 

\bibitem{fk} 
Folland, Gerald   B. and  Kohn, Joseph  J.:   \newblock
{\em The Neumann problem for the Cauchy-Riemann complex.}
Annals of Mathematics Studies, No. {\bf 75.}
Princeton University Press, Princeton, N.J.;
University of Tokyo Press, Tokyo, 1972.

\bibitem{gh} 
Griffith, Phillip and Harris, Joseph:  \newblock
{\em  Principles of Algebraic Geometry.}
John Wiley and Sons;
New York, 1978.


\bibitem{Hor1}   H\"ormander, Lars:  $L^2$ estimates
and existence theorems for the $\bar\partial$
operator. {\em Acta Math.}  {\bf 113}   (1965)  
89-152 


\bibitem{Hor2} H\"{o}rmander, Lars:
The null space of the $\overline\partial$-Neumann operator. 
{\em Ann. Inst. Fourier (Grenoble)} {\bf 54} (2004), no. 5, 1305--1369.

\bibitem{hor}  John Horv\'{a}th: {\em
 Topological vector spaces and distributions. Vol. I.}
Addison-Wesley Publishing Co., Reading, Mass.-London-Don Mills, Ont. 1966 .
 
\bibitem{kr} Kadison, Richard V. and Ringrose, John R.:
{\em Fundamentals of the theory of operator algebras.} Vols I and II.
Reprint of the 1983 original. {\em  Graduate Studies in Mathematics,} {\bf{ 15}} and {\bf{16}.}
American Mathematical Society, Providence, RI, 1997.

\bibitem{Ko1}  Kohn, Joseph J. : Harmonic integrals
on strongly pseudoconvex manifolds, I. {\em  Ann.
of Math.} {\bf 78}  (1963)  112-148 





\bibitem{Ko2}   Kohn, Joseph J.:
  Global regularity for $\overline\partial$ on weakly
pseudoconvex manifolds.
   {\em Trans. Amer. Math. Soc.} {\bf 181}  (1973) 273--292

 
 
\bibitem{lan} Landucci, Mario:
Uniform bounds on derivatives for the $\dbar$-problem in the polydisk. 
{\em  Proceedings Sympos. Pure Math. Vol. XXX}  Part I,  Williamstown, Mass.  (1975) 177-180.
 
 
\bibitem{ms4}Michel, Joachim and Shaw, Mei-Chi:
 The $\overline\partial$ problem on domains with piecewise smooth boundaries with applications.
{\em Trans. Amer. Math. Soc.}{\bf  351}  (1999),  no. 11, 4365--4380.

 

\bibitem{mmt}  Mitrea, Dorina; Mitrea, Marius and Taylor, Michael: Layer potentials, the Hodge Laplacian, and global boundary problems in nonsmooth Riemannian manifolds.  Mem. Amer. Math. Soc.  150  (2001),  no. 713.

\bibitem{mcs1} Shaw, Mei-Chi:    
  Global solvability and regularity for
$\overline\partial$ on an annulus between two weakly
pseudoconvex domains.
{\em Trans. Amer. Math. Soc.}  {\bf 291}  (1985),  255--267.

 \bibitem{mcs2} Shaw, Mei-Chi:  Boundary value problems on Lipschitz domains in $\rl^n$ or $\cx^n$.  {\em Geometric analysis of PDE and several complex variables,}  375--404, Contemp. Math., 368, Amer. Math. Soc., Providence, RI, 2005.
 
\bibitem{mcs3} Shaw, Mei-Chi:  The closed range property for $\dbar$ on domains with pseudoconcave boundary. {\em Proceedings of the Complex Analysis  conference, }  Fribourg, Switzerland,  (2008) 

\bibitem{nick}
Nickerson, H. K:
On the complex form of the Poincar\'{e} lemma.
{\em Proc. Amer. Math. Soc.}{\bf 9} 1958, 183--188.

\bibitem{nw} Nijenhuis, Albert, and Woolf, William B.;
Some integration problems in almost-complex and complex manifolds.
{\em Ann. of Math. (2)}{\bf 77} 1963 424--489. 

\bibitem{scha}
Schaefer, H. H., and Wolff, M. P.,
{\em Topological vector spaces.} 
Second edition. Graduate Texts in Mathematics, 3. Springer-Verlag, New York, 1999. 


\bibitem{weid} 
Weidmann, Joachim:
{\em Linear operators in Hilbert spaces.}
Graduate Texts in Mathematics, 68. Springer-Verlag, New York-Berlin, 1980.


\bibitem{zuck} Zucker, Steven:  $L_{2}$ cohomology of warped products
and arithmetic groups.{\em   Invent. Math.}{\bf  70 } (1982),  no. 2, 169--218
\end{thebibliography}
\end{document}